\documentclass{article}

\usepackage{amsmath}
\usepackage{amssymb}
\usepackage{amsthm}
\usepackage{url}

\newcounter{global}
\theoremstyle{definition}
\newtheorem{definition}[global]{Definition}
\theoremstyle{plain}
\newtheorem{theorem}[global]{Theorem}
\newtheorem{lemma}[global]{Lemma}
\newtheorem{corollary}[global]{Corollary}
\newtheoremstyle{note}{}{}{}{}{\itshape}{.}{.5em}{}
\theoremstyle{note}
\newtheorem{remark}{Remark}%
\newtheorem{example}{Example}%
\renewcommand\section{%
  \@startsection {section}{1}{\z@}%
  {-3.5ex \@plus -1ex \@minus -.2ex}%
  {2.3ex \@plus.2ex}%
  {\normalfont\large\bfseries}}
\pdfmapline{=eurmo10 EURM10 ".167 SlantFont" <eurm10.pfb}
\pdfmapline{=eurmo10 EURM10 " .167 SlantFont" <eurm10.pfb}

\DeclareFontFamily{U}{eurmo}{}
\DeclareFontShape{U}{eurmo}{m}{n}{<-6>eurmo10<6-8>eurmo10<8->eurmo10}{}
\DeclareMathAlphabet{\eurmo}{U}{eurmo}{m}{n}

\def\itm#1{{\rm(\textit{\romannumeral#1})}}

\def\e#1{\eurmo{#1}}

\def\eq{\approx}
\def\ineq{\ensuremath{\preccurlyeq}}
\def\iineq{\succcurlyeq}
\def\tu#1{\left<#1\right>}
\def\br#1{\left[#1\right]}
\newcommand{\st}[1][M]{{\ensuremath{\mathbf{#1}}}}
\newcommand{\alg}[1][M]{\ensuremath{\tu{#1, F^{\st[#1]}}}}
\newcommand{\lalg}[1][M]{\ensuremath{\tu{#1, \eq^{\st[#1]}, F^{\st[#1]}}}}
\newcommand{\llalg}[1][M]{\ensuremath{\tu{#1, \eq^{\st[#1]}, \ineq^{\st[#1]}, F^{\st[#1]}}}}
\newcommand{\loalg}[1][M]{\ensuremath{\tu{#1, \ineq^{\st[#1]}, F^{\st[#1]}}}}

\renewcommand{\L}{{\st[L]}}
\newcommand{\K}{{\cal K}}
\newcommand{\TX}[1][X]{{{\st[T](#1)}}}
\let\dprod=\prod
\newcommand{\dst}[1][M]{{\ensuremath{\dprod_{i \in I}\!\st[#1]_i}}}
\newcommand{\icng}[1][\theta]{\ensuremath{#1}}
\newcommand{\qc}[1][\xi]{\ensuremath{#1}}

\newcommand{\cl}[2][\xi]{{\ensuremath{\br{#2}_{#1}}}}
\newcommand{\fcon}[2][X]{{\xi_{#2}(#1)}}
\newcommand{\val}[2][\st,v]{%
  \ensuremath{\left\lVert%
      \bgroup #2\egroup%
    \right\rVert_{#1}}}
\newcommand{\home}[1]{\ensuremath{{#1^{\scriptscriptstyle\sharp}}}}
\newcommand{\rflc}[1]{\ensuremath{g_{#1}}}
\newcommand{\Mod}[1][\Sigma]{{\mathop{\mathrm{Mod\!}}{(#1)}}}
\newcommand{\IneqX}[2][\K]{\mathop{\mathrm{Ineq\kern-.2pt}_{#2}\kern-2.4pt}{(#1)}}
\newcommand{\falg}[2][X]{{{\st[F]}\kern-.3ex_{#2}(\overline{#1})}}
\newcommand{\oper}[1][O]{{\mathop{\mathrm{#1\!\,}}}}

\def\V{{\oper[V]}}
\newcommand{\fcset}[2][X]{{\Phi_{{\cal #2}}(#1)}}
\def\se#1{\left\{#1\right\}}

\def\smallcap{\mathop{{\textstyle\bigcap}}\nolimits}%
\let\oldphi=\phi
\let\phi=\varphi
\let\varphi=\oldphi
\newcommand{\gen}[2][M]{{\br{#2}_{\st[#1]}}}
\newcommand{\stab}[2][N]{\ensuremath{\e{#2}^{\supseteq #1}}}
\def\ske#1{\mathop{\mathrm{ske\!}}{(#1)}}
\newcommand{\thr}[2][M]{\ensuremath{\st[#1]^{#2}}}

\begin{document}

\title{Variety theorem for algebras with fuzzy order}

\date{\normalsize%
  Dept. Computer Science, Palacky University, Olomouc}

\author{Vilem Vychodil\footnote{%
    e-mail: \texttt{vychodil@binghamton.edu},
    phone: +420 585 634 705,
    fax: +420 585 411 643}}

\maketitle

\begin{abstract}
  We present generalization of the Bloom variety theorem of ordered algebras
  in fuzzy setting. We introduce algebras with fuzzy orders which consist
  of sets of functions which are compatible with particular binary fuzzy
  relations called fuzzy orders. Fuzzy orders are defined on universe sets
  of algebras using complete residuated lattices as structures of degrees.
  In this setting,
  we show that classes of models of fuzzy sets of inequalities are closed
  under suitably defined formations of subalgebras, homomorphic images,
  and direct products. Conversely, we prove that classes having these closure
  properties are definable by fuzzy sets of inequalities.
\end{abstract}

\section{Introduction}\label{sec:intro}
In this paper, we develop previous results on closure properties of model
classes of fuzzy structures~\cite{Be:Bvtafl,BeVy:FHLII} which focused on
algebras equipped with fuzzy equalities. Recall from~\cite{BeVy:Awfs} that
algebras with fuzzy equalities are considered as general algebras, i.e.,
structures consisting of universe sets equipped with (ordinary) $n$-ary
functions, which are in addition compatible with given fuzzy equality relations.
Thus, an algebra with fuzzy equality may be viewed as a traditional algebra
with an additional relational component. The aim of algebras with fuzzy
equalities is to formalize functional systems which preserve similarity
in that the functions, when used with pairwise similar arguments,
produce similar results.
The existing results on closure properties of
algebras with fuzzy equalities include generalization~\cite{Be:Bvtafl} of
the Birkhoff variety theorem~\cite{Bi:Osaa} and results on classes definable
by graded implications between identities~\cite{BeVy:FHLII},
cf. also~\cite{BeVy:FEL} for a survey of results.

Our aim in this paper is to study algebras equipped with fuzzy orders which
in general are intended to formalize different relationships than similarity
fuzzy relations and may be regarded as representations of (graded) preferences.
Analogously as in the case of algebras with fuzzy equalities, we assume that
functions of algebras are compatible with given fuzzy orders. In this case,
the compatibility says that the degree to which the result of a function
applied to $\e{b}_1,\dots,\e{b}_n$ is preferred to the result of the
function applied to $\e{a}_1,\dots,\e{a}_n$ is at least as high as the
degree to which $\e{b}_1$ is preferred to $\e{a}_1$ \emph{and} $\cdots$
\emph{and} $\e{b}_n$ is preferred to $\e{a}_n$ with the logical connective
``\emph{and}'' interpreted by a suitable truth function, e.g.,
a left-continuous triangular norm~\cite{KMP:TN}.

One of the important aspects of our approach is that we consider theories
as fuzzy sets of atomic formulas (inequalities)
prescribing degrees to which the formulas
shall be satisfied in models. We thus utilize one of the basic concepts
of Pavelka's abstract fuzzy logic \cite{Pav:Ofl1,Pav:Ofl2,Pav:Ofl3} in which
theories are fuzzy sets of (abstract formulas) and the semantic entailment
is defined to degrees, cf. also \cite{Ger:FL} for a general treatment of
related topics. In this setting, we prove that model classes of fuzzy sets
of inequalities are closed under formations of homomorphic images,
subalgebras, and direct products and, conversely, classes obeying these
closure properties are model classes of fuzzy sets of inequalities,
establishing an analogy of the Bloom variety theorem~\cite{Bloom}
for ordered algebras.

This paper is organized as follows. In Section~\ref{sec:prelim},
we present preliminaries. In Section~\ref{sec:algs}, we present
the algebras with fuzzy order and their relationship to algebras with
fuzzy equalities. Section~\ref{sec:ex} presents examples of algebras
with fuzzy order. Section~\ref{sec:ac} develops basic algebraic
constructions which are further exploited in Section~\ref{sec:vars}
devoted to the relationship of varieties and inequational classes
of algebras with fuzzy order.

\section{Preliminaries}\label{sec:prelim}
A complete (integral commutative) residuated lattice \cite{Bel:FRS,GaJiKoOn:RL}
is an algebra $\L=\langle L,\wedge,\vee,\otimes,\rightarrow,0,1\rangle$
where
$\langle L,\wedge,\vee,0,1 \rangle$ is a complete lattice,
$\langle L,\otimes,1 \rangle$ is a commutative monoid, and
$\otimes$ and $\rightarrow$ satisfy the adjointness property:
$a \otimes b \leq c$ if{}f $a \leq b \rightarrow c$
($a,b,c \in L$). One of the important properties implied by the
adjointness property is the isotony of $\otimes$. That is,
$a \leq b$ and $c \leq d$ yeild $a \otimes c \leq b \otimes d$.
Examples of complete residuated
lattices include structures on the real unit interval given
by left-continuous t-norms~\cite{KMP:TN} as well as finite structures
and the structures play important role in fuzzy logics in
the narrow sense~\cite{EsGo:MTL,Got:Mfl,Haj:MFL}. A survey of recent
results in fuzzy logics can be found in~\cite{CiHaNo1,CiHaNo2}.

Given $\L$ and $M \ne \emptyset$, a binary $\L$-relation $R$
on $M$ is a map $R\!: M \times M \to L$. For $\e{a},\e{b} \in M$,
the degree $R(\e{a},\e{b}) \in L$ is interpreted as the degree to
which $\e{a}$ and $\e{b}$ are $R$-related. If $\approx,\ineq,\ldots$
denote binary $\L$-relations, we use the usual infix notation and
write $\e{a} \approx \e{b}$ instead of ${\approx}(\e{a},\e{b})$ and the like.
For binary $\L$-relations $R_1$ and $R_2$ on $M$, we put $R_1 \subseteq R_2$
whenever $R_1(\e{a},\e{b}) \leq R_2(\e{a},\e{b})$ for all $\e{a},\e{b} \in M$
and say that $R_1$ is (fully) contained in $R_2$.
As usual, a binary $\L$-relation $R^{-1}$ satisfying
$R^{-1}(\e{a},\e{b}) = R(\e{b},\e{a})$ (for all $\e{a},\e{b} \in M$)
is called the inverse of $R$. For convenience, inverses of relations
denoted by symbols like $\ineq$,\,\ldots\ are denoted by $\iineq$,\,\ldots\
Operations with $\L$-relations
can be defined componentwise~\cite[p. 80]{Bel:FRS} using operations in $\L$.
For instance,
for an $I$-indexed family $\{R_i;\, i \in I\}$ of binary
$\mathbf{L}$-relations on $M$, we may consider the intersection
$\bigcap_{i \in I} R_i$ of all $R_i$ ($i \in I$) which is
a binary $\L$-relation with $\bigl(\bigcap_{i \in I} R_i\bigr)(\e{a},\e{b}) =
\bigwedge_{i \in I}R_i(\e{a},\e{b})$ for all $\e{a},\e{b} \in M$.

In the paper, we consider algebras of a given type. A type is given by
a set $F$ of function symbols together with their arities. We assume that
arity of each $f \in F$ is finite. Recall that an algebra (of type $F$)
is a structure $\st = \alg$ where $M$ is a non-empty universe
set and $F^{\st}$ is a set of functions interpreting the function symbols
in $F$. Namely,
\begin{align*}
  F^{\st} = \bigl\{f^{\st}\!: M^n \to M;\,
  f \text{ is $n$-ary function symbol in } F\bigr\}.
\end{align*}
An algebra with $\L$-equality~\cite[Definition 3.1]{BeVy:Awfs} (of type $F$)
is a structure
$\st = \lalg$ such that $\alg$ is an algebra (of type $F$) and
$\eq^{\st}$ is a binary $\L$-relation on $M$ satisfying the
following conditions:
\begin{align}
  \e{a} \eq^{\st} \e{b} = 1
  &\text{ if{}f } \e{a} = \e{b}, \label{eqn:eRef} \\
  \e{a} \eq^{\st} \e{b}
  &= \e{b} \eq^{\st} \e{a}, \label{eqn:eSym} \\
  \e{a} \eq^{\st} \e{b} \otimes \e{b} \eq^{\st} \e{c}
  &\leq \e{a} \eq^{\st} \e{c}, \label{eqn:eTra} \\
  \e{a}_1 \eq^{\st} \e{b}_1 \otimes \cdots \otimes \e{a}_n \eq^{\st} \e{b}_n
  &\leq f^{\st}(\e{a}_1,\ldots,\e{a}_n) \eq^{\st} f^{\st}(\e{b}_1,\ldots,\e{b}_n)
  \label{eqn:eCom}
\end{align}
for all $\e{a},\e{b},\e{c},\e{a}_1,\e{b}_1,\ldots,\e{a}_n,\e{b}_n \in M$
and any $n$-ary $f \in F$. Recall that condition \eqref{eqn:eRef} represents
reflexivity and separability, \eqref{eqn:eSym} is symmetry, \eqref{eqn:eTra}
is transitivity with respect to $\otimes$ in $\L$
(so-called $\otimes$-transitivity), and \eqref{eqn:eCom} states that $f^{\st}$
is compatible with $\eq^{\st}$ in sense that pairwise similar elements
are mapped by $f^{\st}$ to pairwise similar results.
Algebraic constructions of algebras with fuzzy equalities are developed
in~\cite{BeVy:Awfs,Vy:DLaRPoAwFE},
see also~\cite{BeVy:FEL} for an extensive treatment.

\section{Algebras with Fuzzy Order}\label{sec:algs}
In this section, we introduce algebras with fuzzy order. The motivation is
similar as in the case of algebras with fuzzy equalities. The intention is to
introduce algebras whose functions are compatible with given fuzzy order.
We start by considering algebras with fuzzy equalities as the starting
structures and equip them with additional binary $\L$-relations. Later,
we observe that the additional $\L$-relations in fact determine the
underlying fuzzy equalities.

\begin{definition}\label{def:llalg}
  Let $\st = \lalg$ be an algebra with $\L$-equality. Let $\ineq^{\st}$
  be a binary $\L$-relation on $M$ such that
  \begin{align}
    \e{a} \ineq^{\st} \e{a}
    &= 1, \label{eqn:iRef} \\
    \e{a} \ineq^{\st} \e{b} \wedge \e{b} \ineq^{\st} \e{a}
    &\leq \e{a} \eq^{\st} \e{b}, \label{eqn:iAsy} \\
    \e{a} \ineq^{\st} \e{b} \otimes \e{b} \ineq^{\st} \e{c}
    &\leq \e{a} \ineq^{\st} \e{c}, \label{eqn:iTra} \\
    \e{a}_1 \eq^{\st} \e{b}_1 \otimes \e{a}_2 \eq^{\st} \e{b}_2
    &\leq \e{a}_1 \ineq^{\st} \e{a}_2
    \rightarrow \e{b}_1 \ineq^{\st} \e{b}_2, \label{eqn:iComE} \\
    \e{a}_1 \ineq^{\st} \e{b}_1 \otimes \cdots \otimes
    \e{a}_n \ineq^{\st} \e{b}_n
    &\leq f^{\st}(\e{a}_1,\ldots,\e{a}_n) \ineq^{\st}
    f^{\st}(\e{b}_1,\ldots,\e{b}_n),
    \label{eqn:iComF}
  \end{align}
  for all $\e{a},\e{b},\e{c},\e{a}_1,\e{b}_1,\ldots,\e{a}_n,\e{b}_n \in M$
  and any $n$-ary $f \in F$. The structure $\llalg$ is called
  an algebra with $\L$-equality and $\L$-order.
\end{definition}

As in the case of algebras with fuzzy equalities, \eqref{eqn:iRef} states that
$\ineq^{\st}$ is reflexive, \eqref{eqn:iAsy} represents antisymetry
of $\ineq^{\st}$ with respect to $\eq^{\st}$, \eqref{eqn:iTra}
is $\otimes$-transitivity, and \eqref{eqn:iComF} is compatibility of $f^{\st}$
with $\ineq^{\st}$. Put in words, \eqref{eqn:iComF} states that the degree
to which $f^{\st}(\e{b}_1,\ldots,\e{b}_n)$ is greater than (or equal to)
$f^{\st}(\e{a}_1,\ldots,\e{a}_n)$ is at least the degree to which
$\e{b}_1$ is greater (or equal to) than $\e{a}_1$ and $\cdots$ and 
$\e{b}_n$ is greater (or equal to) than $\e{a}_n$ with ``and'' interpreted
by $\otimes$.
Using the adjointness, \eqref{eqn:iComE} can be equivalently restated as
\begin{align}
  \e{a}_1 \eq^{\st} \e{b}_1 \otimes \e{a}_2 \eq^{\st} \e{b}_2
  \otimes \e{a}_1 \ineq^{\st} \e{a}_2
  &\leq \e{b}_1 \ineq^{\st} \e{b}_2 \label{eqn:iComE_adj}
\end{align}
and the condition expresses that $\ineq^{\st}$ is compatible with $\eq^{\st}$.
It can be shown that the conditions \eqref{eqn:iAsy} and \eqref{eqn:iComE}
we have used together represent a constraint on $\eq^{\st}$ with respect
to $\ineq^{\st}$. Namely, $\eq^{\st}$ in $\st$ is uniquely given
by $\ineq^{\st}$ as it is shown by the following assertion.

\begin{theorem}\label{th:prop_ord}
  Let $\st = \lalg$ be an algebra with $\L$-equality and let $\ineq^{\st}$
  be a binary $\L$-relation on $M$ satisfying
  \eqref{eqn:iRef}, \eqref{eqn:iAsy}, \eqref{eqn:iTra}, and \eqref{eqn:iComF}.
  Then the following conditions are equivalent:
  \begin{enumerate}
  \item[\itm{1}]
    $\llalg$ is an algebra with $\L$-equality and $\L$-order;
  \item[\itm{2}]
    $\eq^{\st}$ and $\ineq^{\st}$ satisfy \eqref{eqn:iComE};
  \item[\itm{3}]
    $\eq^{\st} \mathop{\subseteq} \ineq^{\st}$;
  \item[\itm{4}]
    $\eq^{\st} \mathop{=} \ineq^{\st} \cap \iineq^{\st}$;
  \item[\itm{5}]
    $\eq^{\st}$ is the symmetric interior of $\ineq^{\st}$.
  \end{enumerate}
\end{theorem}
\begin{proof}
  ``\itm{1}\,$\Rightarrow$\,\itm{2}'': Trivial.

  ``\itm{2}\,$\Rightarrow$\,\itm{3}'':
  As a particular case of~\eqref{eqn:iComE} for $\e{a}_1 = \e{a}_2 = \e{b}_1$,
  we have
  \begin{align*}
    \e{b}_1 \eq^{\st} \e{b}_1 \otimes
    \e{b}_1 \eq^{\st} \e{b}_2 \otimes
    \e{b}_1 \ineq^{\st} \e{b}_1 \leq
    \e{b}_1 \ineq^{\st} \e{b}_2.
  \end{align*}
  Now, using~\eqref{eqn:eRef} and~\eqref{eqn:iRef}, the previous inequality
  yields
  \begin{align*}
    \e{b}_1 \eq^{\st} \e{b}_2 =
    1 \otimes \e{b}_1 \eq^{\st} \e{b}_2 \otimes 1
    \leq \e{b}_1 \ineq^{\st} \e{b}_2,
  \end{align*}
  proving \itm{3}.

  ``\itm{3}\,$\Rightarrow$\,\itm{4}'':
  Take any $\e{a},\e{b} \in M$. Since $\eq^{\st}$ satisfies \eqref{eqn:eSym},
  using the assumption $\eq^{\st} \mathop{\subseteq} \ineq^{\st}$, we get
  \begin{align*}
    \e{a} \eq^{\st} \e{b} = \e{b} \eq^{\st} \e{a} \leq
    \e{b} \ineq^{\st} \e{a} = \e{a} \iineq^{\st} \e{b}
  \end{align*}
  which yields
  $\eq^{\st} \mathop{\subseteq} \iineq^{\st}$.
  Hence, $\eq^{\st} \mathop{\subseteq} \ineq^{\st} \cap \iineq^{\st}$.
  The converse inclusion follows directly by~\eqref{eqn:iAsy}.

  ``\itm{4}\,$\Rightarrow$\,\itm{5}'':
  Follows by properties of interior systems and interior operators.
  In a more detail, $\ineq^{\st} \cap \iineq^{\st}$ is obviously
  symmetric and it is a subset of $\ineq^{\st}$. If $R$ is symmetric and
  $R \mathop{\subseteq} \ineq^{\st}$, then the symmetry of $R$ gives 
  $R \mathop{\subseteq} \iineq^{\st}$ and thus
  $R \mathop{\subseteq} \ineq^{\st} \cap \iineq^{\st}$, proving that 
  $\ineq^{\st} \cap \iineq^{\st}$ is the symmetric interior of $\ineq^{\st}$.

  ``\itm{5}\,$\Rightarrow$\,\itm{1}'':
  It suffices to check that $\eq^{\st}$ and $\ineq^{\st}$
  satisfy \eqref{eqn:iComE}. The fact that $\eq^{\st}$ is the symmetric
  interior of $\ineq^{\st}$ yields 
  $\eq^{\st} \mathop{\subseteq} \ineq^{\st}$ and
  $\eq^{\st} \mathop{\subseteq} \iineq^{\st}$.
  Therefore, using the assumption and \eqref{eqn:iTra} twice we get
  \begin{align*}
    \e{a}_1 \eq^{\st} \e{b}_1 \otimes
    \e{a}_2 \eq^{\st} \e{b}_2
    \otimes \e{a}_1 \ineq^{\st} \e{a}_2
    &\leq
    \e{b}_1 \ineq^{\st} \e{a}_1 \otimes
    \e{a}_2 \ineq^{\st} \e{b}_2 \otimes
    \e{a}_1 \ineq^{\st} \e{a}_2 \\
    &\leq \e{b}_1 \ineq^{\st} \e{b}_2
  \end{align*}
  which proves \eqref{eqn:iComE}.
\end{proof}

\begin{remark}
  Note that \cite[Definition 4.42]{Bel:FRS} introduces the notion of
  an $\L$-ordered set as a set $M$ equipped
  with $\eq^{\st}$ and $\ineq^{\st}$ satisfying
  \eqref{eqn:iRef}, \eqref{eqn:iAsy}, \eqref{eqn:iTra},
  and \eqref{eqn:iComE_adj}. Hence, algebras with fuzzy equalities and
  fuzzy orders introduced by Definition~\ref{def:llalg} are 
  $\L$-ordered sets in sense of \cite{Bel:FRS} with an additional
  compatible functional component. Note also that \cite[Lemma 4.45]{Bel:FRS}
  shows that $\L$-ordered sets satisfy the condition \itm{4}
  in Theorem~\ref{th:prop_ord}. The notion of an $\L$-order we utilize in
  this paper is not the only one generalization of a partial order in graded
  setting. Graded order relations or \emph{fuzzy orders} are often defined
  so that the antisymmetry condition~\eqref{eqn:iAsy} is formulated using
  $\otimes$ in place of $\wedge$, see~\cite{BoDeFo:Cfwo}.
\end{remark}

In our approach to algebras with fuzzy equalities and fuzzy orders,
we have started with algebras with fuzzy equalities and expanded the
structures by an additional binary fuzzy relation. The following theorem
shows that we may as well start with $\ineq^{\st}$ satisfying certain
conditions and consider $\eq^{\st}$ a derived fuzzy relation.

\begin{theorem}\label{th:loalg}
  Let $\st = \alg$ be an algebra and $\ineq^{\st}$ be a binary $\L$-relation
  on $M$ satisfying \eqref{eqn:iRef}, \eqref{eqn:iTra}, \eqref{eqn:iComF}.
  If for any distinct elements $\e{a},\e{b} \in M$ we have
  $\e{a} \ineq^{\st} \e{b} \ne 1$ or $\e{b} \ineq^{\st} \e{a} \ne 1$,
  then $\llalg$ with $\eq^{\st} \mathop{=} \ineq^{\st} \cap \iineq^{\st}$
  is an algebra with $\L$-equality and $\L$-order.
\end{theorem}
\begin{proof}
  First, we show that $\lalg$ is an algebra with $\L$-equality.
  Since $\eq^{\st} \mathop{=} \ineq^{\st} \cap \iineq^{\st}$, we get that
  $\eq^{\st}$ satisfies~\eqref{eqn:eSym}. If $\e{a} = \e{b}$ then
  \eqref{eqn:iRef} yields $\e{a} \eq^{\st} \e{b} = 1$. Conversely,
  if $\e{a} \ne \e{b}$, then $\e{a} \ineq^{\st} \e{b} \ne 1$ or
  $\e{b} \ineq^{\st} \e{a} \ne 1$, i.e., $\e{a} \eq^{\st} \e{b} \ne 1$.
  Altogether, $\eq^{\st}$ satisfies~\eqref{eqn:eRef}. Moreover, $\eq^{\st}$
  is $\otimes$-transitive. Indeed, for $\e{a},\e{b},\e{c} \in M$,
  using \eqref{eqn:iTra}, we have
  \begin{align*}
    \e{a} \eq^{\st} \e{b} \otimes \e{b} \eq^{\st} \e{c} &\leq
    \e{a} \ineq^{\st} \e{b} \otimes \e{b} \ineq^{\st} \e{c} \leq
    \e{a} \ineq^{\st} \e{c}
  \end{align*}
  and 
  \begin{align*}
    \e{a} \eq^{\st} \e{b} \otimes \e{b} \eq^{\st} \e{c} &\leq
    \e{b} \ineq^{\st} \e{a} \otimes \e{c} \ineq^{\st} \e{b} \leq
    \e{c} \ineq^{\st} \e{a}
  \end{align*}
  and thus 
  \begin{align*}
    \e{a} \eq^{\st} \e{b} \otimes \e{b} \eq^{\st} \e{c} &\leq
    \e{a} \ineq^{\st} \e{c} \wedge \e{c} \ineq^{\st} \e{a} = 
    \e{a} \eq^{\st} \e{c}
  \end{align*}
  which proves~\eqref{eqn:eTra}. In a similar way, we prove
  that~\eqref{eqn:eCom} is satisfied. Using~\eqref{eqn:iComF},
  \begin{align*}
    \e{a}_1 \eq^{\st} \e{b}_1 \otimes \cdots \otimes \e{a}_n \eq^{\st} \e{b}_n
    &\leq
    \e{a}_1 \ineq^{\st} \e{b}_1 \otimes \cdots \otimes \e{a}_n \ineq^{\st} \e{b}_n
    \\
    &\leq f^{\st}(\e{a}_1,\ldots,\e{a}_n) \ineq^{\st} f^{\st}(\e{b}_1,\ldots,\e{b}_n).
  \end{align*}
  And analogously for $\ineq$ replaced by $\iineq$ which together
  yield~\eqref{eqn:eCom}. As a consequence, $\lalg$ is an algebra
  with $\L$-equality. Now, in order to see that $\llalg$ is an algebra
  with $\L$-equality and $\L$-order, observe that~\eqref{eqn:iAsy} follows
  directly by $\eq^{\st} \mathop{=} \ineq^{\st} \cap \iineq^{\st}$
  and apply Theorem~\ref{th:prop_ord}.
\end{proof}

By virtue of Theorem~\ref{th:loalg}, we may call an algebra $\st = \llalg$
with $\L$-equality and $\L$-order simply an \emph{algebra with $\L$-order}
and write
\begin{align}
  \st = \loalg.
  \label{eqn:loalg}
\end{align}
Whenever we do so, we consider $\eq^{\st}$ to be the symmetric interior
of $\ineq^{\st}$.

\section{Examples}\label{sec:ex}
In this section, we show examples of algebras with fuzzy order. We focus on
examples of structures that are used further in the paper as well as strucures
that naturally appear in fuzzy relational systems.

\begin{example}\label{ex:discr}
  Recall that in the classic case, each algebra can be expanded to an ordered
  algebra by considering identity as the order. Such ordered algebras are usually
  called \emph{discrete}. In our case, for each algebra $\lalg$ with
  $\L$-equality $\eq^{\st}$ we may consider $\st = \llalg$ where
  $\ineq^{\st} \mathop{=} \eq^{\st}$.
  Observe that $\st$ is indeed an algebra with $\L$-order
  since \eqref{eqn:iRef} follows from \eqref{eqn:eRef},
  \eqref{eqn:iAsy} is satisfied trivially,
  \eqref{eqn:iTra} becomes \eqref{eqn:eTra},
  \eqref{eqn:iComF} becomes \eqref{eqn:eCom}, and
  then \eqref{eqn:iComE} follows by Theorem~\ref{th:prop_ord}.
  As an important consequence, each algebra with $\L$-equality can be
  expanded into an algebra with $\L$-order.
\end{example}

\begin{example}\label{ex:triv_cart}
  There are two important special cases of algebras with $\L$-orders which
  result as in Example~\ref{ex:discr}. First, $\st = \llalg$ is called
  \emph{trivial} if $M = \{\emptyset\}$, and both $\eq^{\st}$ and
  $\ineq^{\st}$ are identities. That is, the universe $M$ of $\st$ consists
  of a single element (the empty set $\emptyset$) and
  $\emptyset \eq^{\st} \emptyset =
  \emptyset \ineq^{\st} \emptyset = 1$.
  Second, let $X$ be a set of variables and let $T(X)$ denote the set of all
  terms in variables $X$ (using function symbols in $F$).
  If $T(X) \ne \emptyset$ (e.g., if $F$ contains a nullary function symbol
  or $X$ is non-empty), we denote by $\TX$ the algebra
  $\langle T(X),\eq^{\TX},\ineq^{\TX},F^{\TX}\rangle$ with
  $\L$-equality and $\L$-order where both $\eq^{\st}$ and $\ineq^{\st}$
  are identities, i.e.,
  \begin{align*}
    t \ineq^{\TX} t' &=
    \left\{
      \begin{array}{@{\,}l@{\quad}l@{}}
      1, & \text{if } t = t', \\
      0, & \text{otherwise,}
    \end{array}
    \right.
  \end{align*}
  and analogously for $\eq^{\TX}$. Thus, $\TX$ results from an ordinary
  term algebra by adding $\eq^{\TX}$ and $\ineq^{\TX}$, both being identity
  $\mathbf{L}$-relations on $\TX$. We therefore call $\TX$
  the (\emph{absolutely free}) \emph{term algebra with $\L$-order}
  over variables in $X$. We shall see later in the paper that the trivial
  and term algebras are two borderline cases among algebras with $\L$-order
  (of the same type) in terms of the satisfaction of inequlities and play
  analogous roles as their classic counterparts.
\end{example}

In the next example, we intend to demonstrate that algebras with $\L$-orders
may naturally appear as structures describing entailment degrees of graded
if-then formulas. For the purpose of illustration, we briefly recall fuzzy
attribute implications (FAIs) and their semantic entailment~\cite{BeVy:ADfDwG}.
Consider a non-empty set $Y$ the elements of which are called \emph{attributes.}
Each map $A\!: Y \to L$ (an $\L$-set in $Y$ in the usual terminology of fuzzy
relational systems~\cite{Bel:FRS,Gog:Lic}) prescribes a degree $A(y) \in L$
to each attribute $y \in Y$. The degree $A(y)$ is interpreted as
``a degree to which an object has attribute $y$'', i.e.,
$A$ may be interpreted as a collection of graded attributes of an object.
The set of all maps of the form $A\!: Y \to L$ is denoted by $L^Y$.

For $A,B \in L^Y$, we may introduce
a \emph{degree to which $A$ is included in $B$} by
\begin{align}
  S(A,B) &= \textstyle\bigwedge_{y \in Y}\bigl(A(y) \rightarrow B(y)\bigr),
  \label{eqn:S}
\end{align}
where $\rightarrow$ is the residuum in $\L$.
A \emph{fuzzy} (or \emph{graded}) \emph{attribute implication}
(shortly, a FAI)
is any expression of the from $A \Rightarrow B$, where $A,B \in L^Y$.
The \emph{degree to which $A \Rightarrow B$ is true in $M \in L^Y$},
written $||A \Rightarrow B||_M$ is defined using~\eqref{eqn:S} as follows:
\begin{align}
  ||A \Rightarrow B||_M &= S(A,M) \rightarrow S(B,M).
  \label{eqn:semM}
\end{align}
Assuming that $M$ is a collection of graded attributes of an object $x$,
we may interpret $||A \Rightarrow B||_M$ as the degree to which it is true
that ``if $x$ has all the (graded) attributes from $A$, then $x$ has all
the (graded) attributes from $B$''. Therefore, graded attribute implications
describe if-then dependencies between graded attributes of objects.
Let us note that the approach in~\cite{BeVy:ADfDwG} is more general in that
the definition~\eqref{eqn:semM} is parameterized by a linguistic
hedge~\cite{BeVy:Fcalh,EsGoNo:Hedges,Za:Afstilh}
but we omit this technical issue for the sake of simplicity.

One of the basic problems related to FAIs is semantic entailment of FAIs
from (graded) collections of other FAIs, its characterization, axiomatization,
and other related problems, cf.~\cite{BeVy:ADfDwG}. Denote the set of
all FAIs (in $Y$) by $\mathit{Fml}$ and let $T$ be
a map $T\!: \mathit{Fml} \to L$. Each $T(A \Rightarrow B)$ may be seen
as a degree to which $A \Rightarrow B$ is prescribed by $T$. We put
\begin{align}
  \mathrm{Mod}(T) &= \{M \in L^Y\!;\,
  T(A \Rightarrow B) \leq ||A \Rightarrow B||_T \text{ for all } A,B \in L^Y\}
  \label{eqn:fai_mod}
\end{align}
and call $\mathrm{Mod}(T)$ the set of all models of $T$. Moreover, we put
\begin{align}
  ||A \Rightarrow B||_T &=
  \textstyle\bigwedge_{M \in \mathrm{Mod}(T)}||A \Rightarrow B||_M
  \label{eqn:fai_sement}
\end{align}
and call $||A \Rightarrow B||_T$ the
\emph{degree to which $A \Rightarrow B$ is semantically entailed by~$T$.}
Note that the way we have defined~\eqref{eqn:fai_sement} conforms to the
graded semantics of formulas as it was proposed by
Pavelka~\cite{Pav:Ofl1,Pav:Ofl2,Pav:Ofl3},
cf. also~\cite{Ger:FL} and~\cite[Section~9.2]{Haj:MFL}.
The next example shows that the entailment degrees may be represented
by an algebra with $\L$-order which is related to~\eqref{eqn:fai_sement}.

\begin{example}\label{ex:fai}
  Assuming that $T$ is a set of FAIs in $Y$, we may introduce an algebra
  of classes of semantically equivalent elements in $L^Y$ with
  the $\L$-order corresponding to~\eqref{eqn:fai_sement}. First, we put
  \begin{itemize}
  \item
    $L^{Y\!}/T =
    \bigl\{\cl[T]{A};\, A \in L^Y\bigr\}$, where
  \item
    $\cl[T]{A} =
    \{B \in L^Y\!;\, ||A \Rightarrow B||_T = ||B \Rightarrow A||_T = 1\}$.
  \end{itemize}
  Obviously, $L^{Y\!}/T$ is a partition of $L^Y$ and
  each $\cl[T]{A} \in L^{Y\!}/T$ contains all $B \in L^Y$ which are equivalent
  to $A$ under $T$. Now, consider
  \begin{align*}
    \st_T =
    \langle L^{Y\!}/T, \ineq^{\st_T}, \cup^{\st_T}, 0^{\st_T}, 1^{\st_T}\rangle,
  \end{align*}
  where $0^{\st_T} = \cl[T]{0_Y}$ (with $0_Y(y) = 0$ for all $y \in Y$),
  $1^{\st_T} = \cl[T]{1_Y}$ (with $1_Y(y) = 1$ for all $y \in Y$),
  $\cl[T]{A} \cup^{\st_T} \cl[T]{B} = \cl[T]{A \cup B}$, and
  \begin{align*}
    \cl[T]{A} \ineq^{\st_T} \cl[T]{B} &= ||A \Rightarrow B||_T.
  \end{align*}
  We claim that $\st_T$ is an algebra with $\L$-order. In order to see that,
  we have to check that $\cup^{\st_T}$ and $\ineq^{\st_T}$ are well defined
  and that all conditions required by Theorem~\ref{th:loalg} hold. The fact
  that $\cup^{\st_T}$ is well defined follows directly by the facts that
  $S(A,M) \wedge S(B,M) = S(A \cup B,M)$ is true for all $A,B,M \in L^Y$ and
  that $||A \Rightarrow B||_T = 1$ means $S(A,M) \leq S(B,M)$
  for each $M \in \mathrm{Mod}(T)$.
  In case of $\ineq^{\st_T}$, we show the argument in a more detail:
  For $A' \in \cl[T]{A}$ and $B' \in \cl[T]{B}$,
  we have $||A' \Rightarrow A||_T = 1$ and 
  $||B \Rightarrow B'||_T = 1$ from which it follows that
  \begin{align*}
    1 \otimes ||A \Rightarrow B||_T \otimes 1
    &= ||A' \Rightarrow A||_T \otimes
    ||A \Rightarrow B||_T \otimes
    ||B \Rightarrow B'||_T \\
    &\leq ||A' \Rightarrow B'||_T,
  \end{align*}
  because
  \begin{align}
    ||A \Rightarrow B||_T \otimes ||B \Rightarrow C||_T &\leq
    \textstyle\bigwedge_{M \in \mathrm{Mod}(T)}
    \bigl(||A \Rightarrow B||_M \otimes ||B \Rightarrow C||_M\bigr) \notag \\
    &\leq
    \textstyle\bigwedge_{M \in \mathrm{Mod}(T)}
    \bigl(S(A,M) \rightarrow S(C,M)\bigr) \notag \\
    &= ||A \Rightarrow C||_T
    \label{eqn:prTra}
  \end{align}
  is true for all $A,B,C \in L^Y$.
  Analogously, we prove $||A' \Rightarrow B'||_T \leq ||A \Rightarrow B||_T$.
  As a consequence, $\ineq^{\st_T}$ is well defined. Furthermore, 
  \eqref{eqn:iRef} holds because $||A \Rightarrow A||_T = 1$ is true and
  \eqref{eqn:iTra} is a consequence of~\eqref{eqn:prTra}. In order to check
  \eqref{eqn:iComF}, it suffices to show that
  \begin{align*}
    ||A \Rightarrow B||_T \otimes ||C \Rightarrow D||_T &\leq
    ||A \cup C \Rightarrow B \cup D||_T
  \end{align*}
  which is indeed the case: This follows by similar arguments as before,
  utilizing the fact that $(a \rightarrow b) \otimes (c \rightarrow d) \leq
  (a \wedge c) \rightarrow (b \wedge d)$ holds for all $a,b,c,d \in L$ in
  any residuated lattice~\cite{GaJiKoOn:RL}.
  Altogether, Theorem~\ref{th:loalg} yields that
  $\st_T$ is an algebra with $\L$-order and the induced $\L$-equality is
  \begin{align*}
    \cl[T]{A} \eq^{\st_T} \cl[T]{B} &=
    ||A \Rightarrow B||_T \wedge ||B \Rightarrow A||_T.
  \end{align*}
  Therefore, $\st_T$ may be understood as an algebraic representation of
  the semantic closure of $T$ with $\cl[T]{A} \eq^{\st_T} \cl[T]{B}$ being
  the degree to which $A$ and $B$ are equivalent under $T$ and 
  $\cl[T]{A} \ineq^{\st_T} \cl[T]{B}$ being the degree to which
  $B$ follows from $A$ under $T$.

  Note that $\st_T$ can be further extended, e.g., by considering unary
  operations $\otimes^a$ ($a \in L$) defined by
  $\otimes^a(\cl[T]{A}) = \cl[T]{a{\otimes}A}$ where $a{\otimes}A \in L^Y$
  is defined by $(a{\otimes}A)(y) = a \otimes A(y)$ for all $y \in Y$.
  Interestingly, structures like $\st_T$ can also be defined based on
  syntactic entailment of FAIs and can be regarded as particular
  Lindenbaum algebras with fuzzy order. We refrain from further details
  because we do not want to digress too much from the main subject of
  this paper which are properties of classes of general structures.
\end{example}

\section{Algebraic Constructions}\label{sec:ac}
In this section, we summarize the basic algebraic constructions essential for
establishing an analogy of the variety theorem in Section~\ref{sec:vars}. The
constructions generalize the algebraic constructions presented
in~\cite{BeVy:Awfs} in a way that by omitting the fuzzy order,
we get exactly the algebraic constructions with algebras with fuzzy equalities
as they have been proposed and described in~\cite{BeVy:Awfs}. Detailed
description of operations for ordinary algebras and algebras with
$\L$-equalities and can be found in \cite{BeVy:FEL,BuSa:CUA}.

\subsubsection*{Subalgebras}
Let $\st = \llalg$ be an algebra with $\L$-order. A subuniverse $N$ of $\st$
is a subset of $M$ which is closed under all functions from $F^{\st}$.
If $N \ne \emptyset$ is a subuniverse of $\st$, then
$\st[N] = \llalg[N]$ where
\begin{itemize}
\item
  $\eq^{\st[N]}$ is restriction of $\eq^{\st}$ to $N$
  (i.e., $\e{a} \eq^{\st[N]} \e{b} = \e{a} \eq^{\st} \e{b}$
  for all $\e{a},\e{b} \in N$);
\item
  $\ineq^{\st[N]}$ is restriction of $\ineq^{\st}$ to $N$
  (i.e., $\e{a} \ineq^{\st[N]} \e{b} = \e{a} \ineq^{\st} \e{b}$
  for all $\e{a},\e{b} \in N$);
\item
  each $f^{\st[N]}$ is a restriction of $f^{\st}$ to $N$
  (i.e., $f^{\st[N]}(\e{a}_1,\ldots,\e{a}_n) = f^{\st}(\e{a}_1,\ldots,\e{a}_n)$
  for all $\e{a}_1,\ldots,\e{a}_n \in N$)
\end{itemize}
is called a \emph{subalgebra} of $\st$.

\subsubsection*{Direct Products}
Let $\st_i$ ($i \in I$) be an $I$-indexed family of algebras with $\L$-order
(of the same type $F$). A \emph{direct product} $\dprod_{i \in I}\st_i$ of 
$\st_i$ ($i \in I$) is an algebra with $\L$-order 
$\dprod_{i \in I}\st_i =
\langle \dprod_{i \in I}M_i,
\ineq^{\dprod_{i \in I}{\st_i}},F^{\dprod_{i \in I}{\st_i}}\rangle$
defined on the direct product $\dprod_{i \in I}M_i$ of the universe sets of
all $\st_i$, where the operations are defined componentwise by
\begin{align}
  f^{\dprod_{i \in I}{\st_i}}(\e{a}_1,\ldots,\e{a}_n)(i) &=
  f^{\st_i}\bigl(\e{a}_1(i),\ldots,\e{a}_n(i)\bigr)
\end{align}
for any $n$-ary $f \in F$, $\e{a}_1,\ldots,\e{a}_n \in \dprod_{i \in I}M_i$,
and $i \in I$; and $\ineq^{\dprod_{i \in I}{\st_i}}$ is introduced as follows:
\begin{align}
  \e{a} \ineq^{\dst} \e{b} &=
  \textstyle\bigwedge_{i \in I}
  \e{a}(i) \ineq^{\st[M]_i} \e{b}(i),
  \label{eqn:dirineq}
\end{align}
for all $\e{a},\e{b} \in \dprod_{i \in I}M_i$. It it easily seen
that $\dprod_{i \in I}{\st_i}$ is indeed an algebra with $\L$-order since
it satisfies the conditions of Theorem~\ref{th:loalg}. As a consequence,
for the corresponding $\L$-equality $\eq^{\dprod_{i \in I}\st_i}$, we get
\begin{align}
  \e{a} \eq^{\dst} \e{b} &=
  \textstyle\bigwedge_{i \in I}
  \e{a}(i) \eq^{\st[M]_i} \e{b}(i).
\end{align}
Analogously as in the ordinary setting, the direct product of an empty
family of algebras with $\L$-order is the trivial algebra with $\L$-order,
see Example~\ref{ex:triv_cart}.

\subsubsection*{Homomorphisms and Factor Algebras}
Let $\st$ and $\st[N]$ be algebras with $\L$-orders (of the same type $F$).
A map $h\!: M \to N$ which satisfies equality 
\begin{align}
  h\bigl(f^{\st}(\e{a}_1,\ldots,\e{a}_n)\bigr) &=
  f^{\st[N]}\bigl(h(\e{a}_1),\ldots,h(\e{a}_n)\bigr)
  \label{eqn:morff}
\end{align}
for any $n$-ary $f \in F$ and all $\e{a}_1,\ldots,\e{a}_n \in M$; and 
\begin{align}
  \e{a} \ineq^{\st} \e{b} &\leq h(\e{a}) \ineq^{\st[N]} h(\e{b})
  \label{eqn:morf}
\end{align}
for all $\e{a},\e{b} \in M$
is called a \emph{homomorphism} and is denoted by $h\!: \st \to \st[N]$.
Therefore, homomorphisms are maps which are compatible with the functional
parts of $\st$ and $\st[N]$ and the $\L$-orders of $\st$ and $\st[N]$.
As a consequence, we get that homomorphisms are also compatible with
$\L$-equalities of $\st$ and $\st[N]$ since for any $\e{a},\e{b} \in M$,
it follows that
\begin{align}
  \e{a} \eq^{\st} \e{b} &\leq h(\e{a}) \eq^{\st[N]} h(\e{b}).
\end{align}
A homomorphism $h\!: \st \to \st[N]$ is called an \emph{embedding} whenever
\begin{align}
  \e{a} \ineq^{\st} \e{b} &= h(\e{a}) \ineq^{\st[N]} h(\e{b})
  \label{eqn:embed}
\end{align}
holds for all $\e{a},\e{b} \in M$. If $h\!: \st \to \st[N]$ is
an embedding then 
\begin{align}
  \e{a} \eq^{\st} \e{b} &= h(\e{a}) \eq^{\st[N]} h(\e{b})
\end{align}
for any $\e{a},\e{b} \in M$, which is easy to see,
cf.~\cite[Lemma 4.46]{Bel:FRS}. If $h\!: \st \to \st[N]$ is surjective,
then $\st[N]$ is called a (\emph{homomorphic}) \emph{image} of $\st$.
As usual, surjective embeddings may be called \emph{isomorphisms.}

\begin{remark}\label{rem:compos}
  Note that we may consider a category consisting of algebras with $\L$-orders
  as objects, homomorphisms of algebras with $\L$-orders as arrows, with
  the composition $\circ$ being the usual composition of maps between sets,
  and identity arrows being the usual identity maps. That is, for
  $f\!: \st[M] \to \st[N]$ and $g\!: \st[N] \to \st[P]$, we put
  $(f \circ g)(\e{a}) = g(f(\e{a}))$
  (written in the diagrammatic notation) for all $\e{a} \in M$.
  Clearly, $f \circ g\!: \st[M] \to \st[P]$ is a homomorphism of
  algebras with $\L$-orders. Let us note that surjective homomorphisms
  are exactly the epic arrows and embeddings are monic but there are
  monic arrows which are not embeddings.
\end{remark}

In the paper, we further utilize the fact that homomorphic images may
be represented by factor algebras with $\L$-orders. Consider an
algebra $\st$ with $\L$-order.
A binary $\L$-relation $\qc$ on $M$ is called
an \emph{$\L$-preorder compatible with $\st$} whenever it satisfies
\begin{align}
  \ineq^{\st} &\subseteq \qc,
  \label{eqn:qcRef} \\
  \qc(\e{a},\e{b}) \otimes \qc(\e{b},\e{c}) &\leq \qc(\e{a},\e{c}),
  \label{eqn:qcTra} \\
  \qc(\e{a}_1,\e{b}_1) \otimes \cdots \otimes
  \qc(\e{a}_n,\e{b}_n)
  &\leq
  \qc\bigl(f^{\st}(\e{a}_1,\ldots,\e{a}_n),
  f^{\st}(\e{b}_1,\ldots,\e{b}_n)\bigr),
  \label{eqn:qcCom}
\end{align}
for all $\e{a},\e{b},\e{c},\e{a}_1,\e{b}_1,\ldots,\e{a}_n,\e{b}_n \in M$
and any $n$-ary $f \in F$. It can be shown that compatible $\L$-preorders
on $\st = \llalg$ induce congruences~\cite[Definition 3.3]{BeVy:Awfs}
on $\lalg$. Indeed, for $\icng$ being the
symmetric interior of $\qc$, i.e.,
for $\icng = \qc \cap \qc^{-1}$, we have
\begin{align}
  \eq^{\st} &\subseteq \icng, \\
  \icng(\e{a},\e{b}) &= \icng(\e{b},\e{a}), \\
  \icng(\e{a},\e{b}) \otimes \icng(\e{b},\e{c}) &\leq \icng(\e{a},\e{c}), \\
  \icng(\e{a}_1,\e{b}_1) \otimes \cdots \otimes
  \icng(\e{a}_n,\e{b}_n)
  &\leq
  \icng\bigl(f^{\st}(\e{a}_1,\ldots,\e{a}_n),
  f^{\st}(\e{b}_1,\ldots,\e{b}_n)\bigr).
\end{align}
Hence, according to~\cite{BeVy:Awfs}, we may consider a factor algebra with
$\L$-equality $\st/\icng = \langle M/\icng,\eq^{\st/\icng},F^{\st/\icng}\rangle$
as follows:
\begin{itemize}
\item
  $M/\icng = \bigl\{\cl[\icng]{\e{a}};\, \e{a} \in M\bigr\}$ where
  $\cl[\icng]{\e{a}} = \{\e{b} \in M;\, \icng(\e{a},\e{b}) = 1\}$;
\item
  $f^{\st/\icng}\bigl(\cl[\icng]{\e{a}_1},\ldots,\cl[\icng]{\e{a}_n}\bigr) =
  \cl[\icng]{f^{\st}(\e{a}_1,\ldots,\e{a}_n)}$;
\item
  $\cl[\icng]{\e{a}} \eq^{\st/\icng} \cl[\icng]{\e{b}} = \icng(\e{a},\e{b})$.
\end{itemize}
That is, $\st/\icng$ can be seen as an ordinary factor algebra $\alg$ modulo
$\{\langle\e{a},\e{b}\rangle;\, \icng(\e{a},\e{b}) = 1\}$ which is in addition
equipped with $\eq^{\st/\icng}$ defined as above.
We put $M/\qc = M/\icng$, $\cl{\e{a}} = \cl[\icng]{\e{a}}$,
$f^{\st/\qc} = f^{\st/\icng}$, $\eq^{\st/\qc} \mathop{=} \eq^{\st/\icng}$ and 
define $\ineq^{\st/\qc}$ by
\begin{align}
  \cl{\e{a}} \ineq^{\st/\qc} \cl{\e{b}} = \qc(\e{a},\e{b}).
  \label{eqn:ineq_fac}
\end{align}
The resulting structure $\st/\qc = \langle M/\qc,\eq^{\st/\qc},
\ineq^{\st/\qc},F^{\st/\qc}\rangle$ is called
a \emph{factor algebra with $\L$-order} of $\st$ modulo $\qc$.
The following lemma shows that $\st/\qc$ is indeed an algebra with $\L$-order:

\begin{lemma}
  $\ineq^{\st/\qc}$ defined by~\eqref{eqn:ineq_fac} is well defined
  and satisfies~\eqref{eqn:iRef}--\eqref{eqn:iComF} with respect
  to $\eq^{\st/\qc}$. As a consequence, $\st/\qc$ is a well-defined
  algebra with $\L$-order.
\end{lemma}
\begin{proof}
  We first show that the value of $\cl{\e{a}} \ineq^{\st/\qc} \cl{\e{b}}$
  does not depend on the choice of elements from either class. Thus,
  take $\e{a}' \in \cl{\e{a}}$ and $\e{b}' \in \cl{\e{b}}$. By definition,
  we get $\icng(\e{a},\e{a}') = 1$ for $\icng = \qc \cap \qc^{-1}$ and
  thus $\qc(\e{a},\e{a}') = 1$ and $\qc(\e{a}',\e{a}) = 1$. Analogously,
  $\qc(\e{b},\e{b}') = 1$ and $\qc(\e{b}',\e{b}) = 1$.
  Applying~\eqref{eqn:qcTra} twice, we get
  \begin{align*}
    1 \otimes 1 \otimes \qc(\e{a},\e{b}) =
    \qc(\e{a}',\e{a}) \otimes \qc(\e{b},\e{b}') \otimes \qc(\e{a},\e{b}) \leq
    \qc(\e{a}',\e{b}')
  \end{align*}
  and analogously $\qc(\e{a}',\e{b}') \leq \qc(\e{a},\e{b})$,
  meaning $\qc(\e{a},\e{b}) = \qc(\e{a}',\e{b}')$, i.e.,
  \eqref{eqn:ineq_fac} is well defined.
  Now, observe that \eqref{eqn:iRef} follows by~\eqref{eqn:qcRef}.
  Moreover, \eqref{eqn:iAsy} follows by the fact that
  \begin{align*}
    \cl[\qc]{\e{a}} \eq^{\st/\qc} \cl[\qc]{\e{b}} = \icng(\e{a},\e{b}) 
    &=
    \qc(\e{a},\e{b}) \wedge \qc(\e{b},\e{a}) \\
    &= \cl[\qc]{\e{a}} \ineq^{\st/\qc} \cl[\qc]{\e{b}}
    \wedge
    \cl[\qc]{\e{b}} \ineq^{\st/\qc} \cl[\qc]{\e{a}},
  \end{align*}
  i.e., $\eq^{\st/\qc} \mathop{=} \ineq^{\st/\qc} \cap \iineq^{\st/\qc}$.
  In addition to that, \eqref{eqn:iTra} follows by \eqref{eqn:qcTra} and
  \eqref{eqn:iComF} follows by \eqref{eqn:qcCom}. Finally,
  Theorem~\ref{th:prop_ord} yields \eqref{eqn:iComE}.
\end{proof}

\begin{remark}
  Let us note that $\st_T$ in Example~\ref{ex:fai} may be seen as
  a factorization of
  $\st = \langle L^Y,\ineq^{\st},\cup^{\st},0_Y,1_Y\rangle$,
  where $A \ineq^{\st} B = S(B,A)$ and $(A \cup^{\st} B)(y) = A(y) \vee B(y)$
  (a union of $\L$-sets defined componentwise using $\vee$ in $\L$).
  In particular, we may take a compatible $\L$-preorder $\qc$ induced by $T$,
  i.e., $\qc(A,B) = ||A \Rightarrow B||_T$. Under this notation, $\st_T$
  coincides with $\st/\qc$.
\end{remark}

Given an $\L$-preorder $\qc$ which is compatible with $\st$, we may introduce
a surjective map $h_{\qc}\!: M \to M/\qc$ by putting
\begin{align}
  h_{\qc}(\e{a}) &= \cl{\e{a}}.
  \label{eqn:nat}
\end{align}
The map is called a \emph{natural homomorphism} induced by $\qc$ and it
is indeed a homomorphism which is in addition surjective. Thus, factor
algebras with $\L$-orders are homomorphic images. Moreover,
for a homomorphism $h\!: \st \to \st[N]$, we may introduce
a binary $\L$-relation $\qc_h$ on $M$ by putting
\begin{align}
  \qc_h(\e{a},\e{b}) &= h(\e{a}) \ineq^{\st[N]} h(\e{b}).
  \label{eqn:hom_th}
\end{align}
Analogously as in case of algebras with $\L$-equalities~\cite{BeVy:Awfs},
one may check that $\qc_h$ is a compatible $\L$-preorder on $\st$ and 
$\st/\qc_h$ can be embedded into $\st[N]$. Namely, one can introduce an
embedding $g\!: \st/\qc_h \to \st[N]$ by putting
$g(\cl[\qc_h]{\e{a}}) = h(\e{a})$ and show that $h = h_{\qc_h} \circ g$.
In addition, if $h$ is surjective then $g$ is an isomorphism,
cf.~\cite[Theorem 3.8]{BeVy:Awfs}.
As a result, homomorphic images of $\st$ are isomorphic to
factor algebras of $\st$
which is a desirable property.

\section{Varieties of Algebras with Fuzzy Order}\label{sec:vars}
In this section, we provide an analogy of the ordered variety theorem by
Bloom~\cite{Bloom}. We prove an assertion showing that classes of algebras
with fuzzy order defined by graded inequalities are exactly the classes of
algebras with fuzzy order closed under the formations of subalgebras,
homomorphic images, and direct products.

We start with proving the closure properties of classes definable by graded
inequalities. Collections of graded inequalities are formalized as
$\L$-sets of formulas. In a more detail, an atomic formula $t \ineq t'$
where $t,t' \in T(X)$ are terms, is called an \emph{inequality.} Then,
an $\L$-set $\Sigma$ of inequalities, called a \emph{theory},
can be seen as a representation of a collection of graded inequalities
using terms in $T(X)$. Indeed, for each $t,t' \in T(X)$, $\Sigma$ prescribes
a degree $\Sigma(t \ineq t')$ which can be interpreted as a lower bound of
a degree to which $t \ineq t'$ shall be satisfied in a model. Clearly,
the standard understanding of theories as sets of formulas can be viewed
as a particular case of the concept of theories as $\L$-sets of formulas
since $\Sigma(t \ineq t') = 1$ prescribes that $t \ineq t'$ shall be satisfied
(fully) in a model of $\Sigma$ and $\Sigma(t \ineq t') = 0$ means that
in a model of $\Sigma$ the inequality $t \ineq t'$ need not be satisfied
at all. In other words, $\Sigma(t \ineq t') = 0$ puts no constraint on
models with respect to the satisfaction of $t \ineq t'$.

In order to define models precisely, we utilize the following notation.
For an algebra $\st$ with $\L$-order, any map $v\!: X \to M$ is called
an $\st$-valuation of variables in $X$, i.e., the result $v(x)$ is the
value of $x$ in $\st$ under $v$. As in the ordinary case, $v$ admits
a unique homomorphic extension $\home{v}$ to $\TX$. Thus, $\home{v}(x) = v(x)$
for all $x \in X$ and $\home{v}$ is a map $\home{v}\!: \TX \to \st$.
For $t \in T(X)$, the value $\home{v}(t)$ is called
the value of $t$ in $\st$ under $v$. Note that in the literature on fuzzy
logics in the narrow sense~\cite{Bel:FRS,BeVy:FEL,Haj:MFL},
$\home{v}(t)$ is often denoted by $\val{t}$.
Now, for any inequality $t \ineq t'$, we may introduce the degree
to which $t \ineq t'$ is true in $\st$ under $v$ by
\begin{align}
  \val{t \ineq t'} &= \home{v}(t) \ineq^{\st} \home{v}(t')
  \label{eqn:valMv}
\end{align}
and the degree to which $t \ineq t'$ is true in $\st$
(under all $\st$-valuations):
\begin{align}
  \val[\st]{t \ineq t'} &= \textstyle\bigwedge_{v: X \to M}\val{t \ineq t'}.
  \label{eqn:valM}
\end{align}
Under this notation, we call $\st$ a model of $\Sigma$ if
$\Sigma(t \ineq t') \leq \val[\st]{t \ineq t'}$ for all $t,t' \in T(X)$ and
denote the class of all models of $\Sigma$ by $\Mod$. That is,
\begin{align}
  \Mod &= \bigl\{\st;\, \Sigma(t \ineq t') \leq \val[\st]{t \ineq t'}
  \text{ for all } t,t' \in T(X)\bigr\}.
\end{align}
A class $\K$ of algebras with $\L$-orders
is called an \emph{inequational class} if there is
$\Sigma$ such that $\K = \Mod$.
Furthermore, if $\K$ is a class of algebras with $\L$-orders, we define
a degree to which $t \ineq t'$ is true in (all algebras in) $\K$ by
\begin{align}
  \val[\K]{t \ineq t'} &= \textstyle\bigwedge_{\st \in \K}\val[\st]{t \ineq t'}.
  \label{eqn:valK}
\end{align}
In particular, for $\K = \Mod$, $\val[\Mod]{t \ineq t'}$ defined
by \eqref{eqn:valK} is called a degree to which $t \ineq t'$ is semantically
entailed by $\Sigma$.

\begin{remark}
  The way we consider theories as $\L$-sets of formulas, define models
  of theories, and define degrees of semantic entailment is consistent with
  Pavelka's general approach to abstract logic
  semantics \cite{Pav:Ofl1,Pav:Ofl2,Pav:Ofl3}. Clearly, if $\st$ is
  the trivial algebra with $\L$-order, then $\val[\st]{t \ineq t'} = 1$
  for any $t,t' \in T(X)$. In contrast, $\val[\TX]{t \ineq t'} = 1$ if
  $t = t'$ and $\val[\TX]{t \ineq t'} = 0$ otherwise, i.e.,
  all non-trivial inequalities are satisfied in $\TX$ to a zero degree.
\end{remark}

The following assertion shows the basic closure properties of $\Mod$.
The proof of the assertion uses analogous arguments as in the case of equational
classes, i.e., the classes of algebras with fuzzy equalities defined by
$\L$-sets of identities~\cite{Be:Bvtafl}.
We therefore present only a sketch of the proof,
omitting technical details.

\begin{theorem}\label{thm:ineq_var}
  Let $\Sigma$ be an $\L$-set of inequalities.
  Then, $\Mod$ is closed under formations of subalgebras,
  homomorphic images, and direct products.
\end{theorem}
\begin{proof}
  The fact that $\Mod$ is closed under formations of subalgebras is immediate
  since $\Sigma(t \ineq t') \leq \val[\st]{t \ineq t'} \leq
  \val[{\st[N]}]{t \ineq t'}$ for each subalgebra $\st[N]$ of $\st \in \Mod$.
  Analogously, if $\st_i \in \Mod$ ($i \in I$), we may use
  $\Sigma(t \ineq t') \leq
  \bigwedge_{i \in I}\val[\st_i]{t \ineq t'} \leq
  \val[\dprod_{i \in I}\st_i]{t \ineq t'}$
  to prove that $\dprod_{i \in I}\st_i \in \Mod$. Indeed, in this case we may
  use the fact that for each $\dprod_{i \in I}\st_i$-valuation $v$ there are
  $\st_i$-valuations $v_i$ ($i \in I$) such that $v_i(x) = (v(x))(i)$ for all
  $x \in X$ and $i \in I$ and thus $\home{v_i}(t) = (\home{v}(t))(i)$
  for all $t \in T(X)$ and $i \in I$, cf.~\cite{BeVy:FEL}.

  Finally, let $\st \in \Mod$, take a surjective homomorphism
  $h\!: \st \to \st[N]$, and consider an $\st[N]$-valuation $v$.
  Since $h$ is surjective, there is an $\st$-valuation $w$ such that
  $v(x) = h(w(x))$ for all $x \in X$,
  i.e., $v = w \circ h$ with $\circ$ defined as
  in Remark~\ref{rem:compos}. By induction on the rank of terms,
  \eqref{eqn:morff} can be used to prove that
  $\home{w} \circ h = \home{(w \circ h)}$. Now, using \eqref{eqn:morf},
  it follows that
  \begin{align*}
    \val[\st,w]{t \ineq t'} &=
    \home{w}(t) \ineq^{\st} \home{w}(t')
    \\
    &\leq
    h(\home{w}(t)) \ineq^{\st[N]} h(\home{w}(t'))
    \\
    &=
    (\home{w} \circ h)(t) \ineq^{\st[N]} (\home{w} \circ h)(t')
    \\
    &=
    \home{(w \circ h)}(t) \ineq^{\st[N]} \home{(w \circ h)}(t')
    \\
    &= 
    \home{v}(t) \ineq^{\st[N]} \home{v}(t') = 
    \val[{\st[N],v}]{t \ineq t'},
  \end{align*}
  showing $\Sigma(t \ineq t') \leq \val[\st]{t \ineq t'} \leq
  \val[{\st[N]}]{t \ineq t'}$. Hence, $\st[N] \in \Mod$.
\end{proof}

Using the standard terminology in general algebra, a class of 
algebras with fuzzy orders is called a \emph{variety} if it is
closed under formations of subalgebras,
homomorphic images, and direct products.
In addition, by $\V(\K)$ we denote the variety generated by $\K$, i.e., 
$\V(\K)$ is the least variety containing each algebra with fuzzy order
in $\K$. The previous theorem states that each inequational class is
a variety. We now turn our attention to the opposite direction of
this observation.

Analogously as in the ordinary case, in order to prove that a variety $\K$
is an inequational classes, we utilize particular algebras in $\K$ which
may be regarded as the greatest images of $\TX$ in $\K$. In particular,
for a class $\K$ of algebras with $\L$-order and for a set $X$ of variables,
we define
\begin{align}
  \fcset{\K} = \se{\qc_h;\, h\!: \TX \to \st \text{ for some } \st \in \K},
  \label{eqn:fcset}
\end{align}
i.e. $\fcset{\K}$ is the set of all compatible $\L$-preorders on $\TX$
induced by homomorphisms from $\TX$ to algebras in $\K$, see~\eqref{eqn:hom_th}.
Using the same arguments as in \cite[Theorem 3.1]{BeVy:Awfs},
the set of all compatible $\L$-preorders on $\TX$ is closed under arbitrary
intersections. Thus, $\fcon{\K} = \smallcap{\fcset{\K}}$ is a compatible
$\mathbf{L}$-preorder on $\TX$. The corresponding factor algebra
\begin{align}
  \falg{\K}=\TX/\fcon{\K}
\end{align}
is called the \emph{$\K$-free algebra with $\L$-order}
(\emph{generated by $X$}).
The following assertions show properties of $\falg{\K}$.

\begin{lemma}
  If $\K$ is variety, then $\falg{\K} \in \K$.
\end{lemma}
\begin{proof}
  We prove the assertion by showing that $\falg{\K}$ is isomorphic
  to a subalgebra of a direct product of algebras $\TX/\qc$ in $\K$.
  Consider a map $h$ of the form
  $h\!: T(X)/\fcon{\K} \to \dprod_{\qc \in \fcset{\K}}T(X)/\qc$
  such that
  \begin{align}
    \bigl(h\bigl(\cl[\fcon{\K}]{t}\bigr)\bigr)(\qc) = \cl[\qc]{t}
    \label{eqn:falg_in_K}
  \end{align}
  for all $t \in T(X)$ and $\qc \in \fcset{\K}$.
  First, the map is well defined
  since~\eqref{eqn:falg_in_K} does not depend on the choice of $t$
  in $\cl[\fcon{\K}]{t}$. In order to see that, observe that for
  $t' \in \cl[\fcon{\K}]{t}$ we have $(\fcon{\K})(t,t') = (\fcon{\K})(t',t) = 1$
  which means $\qc(t,t') = \qc(t',t) = 1$ for all $\qc \in \fcset{\K}$
  and thus $\cl{t} = \cl{t'}$. Moreover, $h$ is compatible with
  functions: Take an $n$-ary $f$, terms $t_1,\ldots,t_n$ and observe
  that for all $\qc \in \fcset{\K}$, we have
  \begin{align*}
    &\bigl(h\bigl(f^{\falg{\K}}\bigl(\cl[\fcon{\K}]{t_1},\ldots,
    \cl[\fcon{\K}]{t_n}\bigr)\bigr)\bigr)(\qc) \\
    &= \bigl(h\bigl(\cl[\fcon{\K}]{f(t_1,\ldots,t_n)}\bigr)\bigr)(\qc) \\
    &= \cl{f(t_1,\ldots,t_n)} \\
    &= f^{\TX/\qc}\bigl(\cl{t_1},\ldots,\cl{t_n}\bigr) \\
    &= f^{\TX/\qc}\bigl(\bigl(h\bigl(\cl[\fcon{\K}]{t_1}\bigr)\bigr)(\qc),
    \ldots,\bigl(h\bigl(\cl[\fcon{\K}]{t_n}\bigr)\bigr)(\qc)\bigr) \\
    &= \bigl(f^{\dprod_{\qc \in \fcset{\K}}\TX/\qc}
    \bigl(h\bigl(\cl[\fcon{\K}]{t_1}\bigr),
    \ldots,h\bigl(\cl[\fcon{\K}]{t_n}\bigr)\bigr)\bigr)(\qc).
  \end{align*}
  Therefore, $h$ satisfies~\eqref{eqn:morff}. In order to finish the proof,
  it suffices to check~\eqref{eqn:embed}. Using~\eqref{eqn:dirineq}, for
  any $t,t' \in T(X)$ it follows that
  \begin{align*}
    \cl[\fcon{\K}]{t}
    \ineq^{\falg{\K}}
    \cl[\fcon{\K}]{t'}
    &= (\fcon{\K})(t,t') \\
    &= \textstyle\bigwedge_{\qc \in \fcset{\K}} \qc(t,t') \\
    &= \textstyle\bigwedge_{\qc \in \fcset{\K}} 
    \cl{t} \ineq^{\TX/\qc} \cl{t'} \\
    &= \textstyle\bigwedge_{\qc \in \fcset{\K}} 
    \bigl(h\bigl(\cl[\fcon{\K}]{t}\bigr)\bigr)(\qc)
    \ineq^{\TX/\qc}
    \bigl(h\bigl(\cl[\fcon{\K}]{t'}\bigr)\bigr)(\qc) \\
    &= h\bigl(\cl[\fcon{\K}]{t}\bigr)
    \ineq^{\dprod_{\qc \in \fcset{\K}}\TX/\qc}
    h\bigl(\cl[\fcon{\K}]{t'}\bigr),
  \end{align*}
  which proves that $h$ is an embedding.
  Therefore, $\falg{\K} \in \K$.
\end{proof}

The following lemma shows that if $\K$ is a variety,
then the natural homomorphism $h_{\fcon{\K}}\!: \TX \to \falg{\K}$
defined by~\eqref{eqn:nat} is a sur-reflection of $\TX$ in $\K$.
As a consequence, $\falg{\K}$ can be seen as the greatest image
of $\TX$ in $\K$, see~\cite{BeVy:FHLII}.

\begin{lemma}\label{le:surref}
  Let $\K$ be a variety, $\st \in \K$, and let $h\!: \TX \to \st$ be
  a homomorphism. Then, there is a unique homomorphism
  $\rflc{h}\!: \falg{\K} \to \st$ such that $h = h_{\fcon{\K}} \circ \rflc{h}$.
\end{lemma}
\begin{proof}
  Consider $\rflc{h}$ given by
  \begin{align*}
    \rflc{h}\bigl(\cl[\fcon{\K}]{t}\bigr) &= h(t)
  \end{align*}
  for all terms $t \in T(X)$. Note that $\rflc{h}$ is well defined. Indeed,
  for $t' \in \cl[\fcon{\K}]{t}$,
  we have $\fcon{\K}(t,t') = \fcon{\K}(t',t) = 1$.
  Since $\qc_h \in \fcset{\K}$, we get
  $\qc_h(t,t') = \qc_h(t',t) = 1$ which yields
  $h(t) \ineq^{\st} h(t') = h(t') \ineq^{\st} h(t) = 1$
  and so
  $h(t) = h(t')$, i.e., the values of $\rflc{h}$ do not depend on the
  terms selected from the classes $\cl[\fcon{\K}]{{\cdots}}$.
  Obviously, $h = h_{\fcon{\K}} \circ \rflc{h}$. We show that $\rflc{h}$
  is a homomorphism.
  For any $n$-ary $f$ and $t_1,\ldots,t_n \in T(X)$, we have
  \begin{align*}
    \rflc{h}\bigl(f^{\falg{\K}}\bigl(\cl[\fcon{\K}]{t_1},\ldots,
    \cl[\fcon{\K}]{t_n}\bigr)\bigr) &=
    \rflc{h}\bigl(\cl[\fcon{\K}]{f(t_1,\ldots,t_n)}\bigr) \\
    &= h(f(t_1,\ldots,t_n)) \\
    &= f^{\st}(h(t_1),\ldots,h(t_n)) \\
    &= f^{\st}\bigl(\rflc{h}\bigl(\cl[\fcon{\K}]{t_1}\bigr),\ldots,
    \rflc{h}\bigl(\cl[\fcon{\K}]{t_n}\bigr)\bigr).
  \end{align*}
  It remains to show that $\rflc{h}$ is compatible with the $\L$-orders
  in $\falg{\K}$ and $\st$. By definition of $\rflc{h}$, it suffices to
  prove
  \begin{align*}
    (\fcon{\K})(t,t') =
    \cl[\fcon{\K}]{t} \ineq^{\falg{\K}} \cl[\fcon{\K}]{t'} \leq
    h(t) \ineq^{\st} h(t') = 
    \qc_h(t,t')
  \end{align*}
  which is indeed the case because $\qc_h \in \fcset{\K}$. Note that
  since $h_{\fcon{\K}}$ is surjective, $\rflc{h}$ is determined uniquely.
\end{proof}

The following lemma is a consequence of the previous observations and shows
that each variety $\K$ of algebras with $\L$-order is generated by
a single $\K$-free algebra with $\L$-order.
Namely, it is sufficient to consider $\falg{\K}$
generated by a denumerable set $X$ of variables.

\begin{lemma}\label{le:V_falg}
  If $\K$ is variety, then $\K = \V(\{\falg{\K}\})$ where $X$ is denumerable.
\end{lemma}
\begin{proof}
  Take $\st \in \K$ and let $\gen{N}$ denote the subalgebra of $\st$ which
  is generated by $N \subseteq M$. We first show that each finitely generated
  subalgebra of $\st$ is an image of $\falg{\K}$ where $X$ is denumerable.

  Let $N \subseteq M$ be finite and take any surjective map $h\!: X \to N$.
  The map can be extended to a surjective homomorphism
  $\home{h}\!: \TX \to \gen{N}$. By Lemma~\ref{le:surref},
  there is a homomorphism $\rflc{\home{h}}\!: \falg{\K} \to \gen{N}$
  such that $\home{h} = h_{\fcon{\K}} \circ \rflc{\home{h}}$.
  In addition, the surjectivity of $\rflc{\home{h}}$ follows directly by
  the surjectivity of $\home{h}$.
  Hence, $\gen{N} \in \V(\{\falg{\K}\})$.

  We continue the proof by showing that $\st$ can be embedded into
  a factorization of a subalgebra of a direct product
  of its finitely generated subalgebras.
  In the proof, we use the following notation: We put
  \begin{align*}
    \mathcal{M} = \{N \ne \emptyset;\,
    N \subseteq M \text{ and } N \text{ is finite}\}
  \end{align*}
  and consider the direct product
  $\textstyle\dprod_{N \in \mathcal{M}}\gen{N}$.
  For any $\e{a} \in \textstyle\dprod_{N \in \mathcal{M}}\gen{N}$,
  we write $\stab{a} = \e{m}$ if the following condition is satisfied:
  $N \in \mathcal{M}$ and for each $N' \in \mathcal{M}$ such that
  $N \subseteq N'$, we have $\e{a}(N') = \e{m}$. Now, let
  \begin{align*}
    P = \{\e{a} \in \textstyle\dprod_{N \in \mathcal{M}}\gen{N};\,
    \stab{a} = \e{m} 
    \text{ for some } N \in \mathcal{M} \text{ and } \e{m} \in M\}.
  \end{align*}
  Clearly, $P$ is a non-empty subuniverse of the direct product of all
  finitely generated subalgebras of $\st$. Furthermore, we may introduce
  a binary $\L$-relation $\qc$ on $\st[P]$ as follows:
  \begin{align*}
    \qc(\e{a},\e{b}) &= \e{m} \ineq^{\st} \e{n}
  \end{align*}
  where $\e{m},\e{n} \in M$ such that
  $\stab{\e{a}} = \e{m}$ and $\stab{\e{b}} = \e{n}$
  for some $N \in \mathcal{M}$.
  Observe that such $N$ always exists and $\e{m},\e{n} \in M$
  are given uniquely. Thus, $\qc$ is well defined.
  In addition to that, $\qc$ is a compatible $\L$-preorder on $\st[P]$:
  Reflexivity of $\qc$ is obvious;
  in order to see that $\qc$ is $\otimes$-transitive,
  consider $\e{a},\e{b},\e{c} \in P$ and let 
  $\qc(\e{a},\e{b}) = \e{m} \ineq^{\st} \e{n}$ for
  $\stab{\e{a}} = \e{m}$ and $\stab{\e{b}} = \e{n}$ and
  $\qc(\e{b},\e{c}) = \e{p} \ineq^{\st} \e{q}$ for
  $\stab[N']{\e{b}} = \e{p}$ and $\stab[N']{\e{c}} = \e{q}$.
  Thus, for $N \cup N'$, we have $\stab[N \cup N']{\e{a}} = \e{m}$,
  $\stab[N \cup N']{\e{c}} = \e{q}$, and 
  $\stab[N \cup N']{\e{b}} = \e{n} = \e{p}$.
  Hence, $\qc(\e{a},\e{b}) \otimes \qc(\e{b},\e{c}) =
  \e{m} \ineq^{\st} \e{n} \otimes \e{n} \ineq^{\st} \e{q} \leq
  \e{m} \ineq^{\st} \e{q} = \qc(\e{a},\e{c})$. Analogously one
  may check that $\qc$ is compatible with all functions in $\st[P]$.
  Moreover, $\ineq^{\st[P]} \subseteq \qc$ follows directly
  by the fact that $\st[P]$ is a subalgebra of 
  $\textstyle\dprod_{N \in \mathcal{M}}\gen{N}$.

  Finally, consider $h\!: M \to P/\qc$ such that
  $h(\e{m}) = \cl{\e{a}}$ for $\stab{\e{a}} = \e{m}$. Clearly,
  $h$ is well defined since for $\stab[N']{\e{b}} = \e{m}$, we have
  $\qc(\e{a},\e{b}) = \qc(\e{b},\e{a}) = 1$
  and thus $\cl{\e{a}} = \cl{\e{b}}$.
  Now, take an $n$-ary function $f^{\st}$
  and $\e{m}_1,\ldots,\e{m}_n \in M$.
  We have
  \begin{align*}
    f^{\st[P]/\qc}\bigl(h(\e{m}_1),\ldots,h(\e{m}_n)\bigr) &= 
    f^{\st[P]/\qc}\bigl(\cl{\e{a}_1},\ldots,\cl{\e{a}_n}\bigr) =
    \cl{f^{\st[P]}(\e{a}_1,\ldots,\e{a}_n)},
  \end{align*}
  where $\stab[N_1]{a}_1 = \e{m}_1,\ldots,\stab[N_n]{a}_n = \e{m}_n$.
  Observe that for $\e{a} = f^{\st[P]}(\e{a}_1,\ldots,\e{a}_n)$
  we have $\stab[N_1 \cup \cdots \cup N_n]{a} =
  f^{\st}(\e{m}_1,\ldots,\e{m}_n)$ and thus
  \begin{align*}
    f^{\st[P]/\qc}\bigl(h(\e{m}_1),\ldots,h(\e{m}_n)\bigr) &= 
    \cl{f^{\st[P]}(\e{a}_1,\ldots,\e{a}_n)} =
    h\bigl(f^{\st}(\e{m}_1,\ldots,\e{m}_n)\bigr).
  \end{align*}
  We now show that $h$ is an embedding. Take $\e{m},\e{n} \in M$
  and consider $\e{a},\e{b} \in P$ such that $\stab{\e{a}} = \e{m}$ and
  $\stab[N']{\e{b}} = \e{n}$, respectively. Since for $N \cup N'$ we have
  $\stab[N \cup N']{\e{a}} = \e{m}$ and $\stab[N \cup N']{\e{b}} = \e{n}$,
  it follows that
  \begin{align*}
    \e{m} \ineq^{\st} \e{n} &= \qc(\e{a},\e{b}) =
    \cl{\e{a}} \ineq^{\st[P]/\qc} \cl{\e{b}} =
    h(\e{m}) \ineq^{\st[P]} h(\e{n}).
  \end{align*}
  Hence, the map $h\!: \st \to \st[P]/\qc$ defined as above is an embedding.
  As a consequence, we have $\st \in \V(\{\gen{N};\, N \in \mathcal{M}\})
  \subseteq \V(\{\falg{\K}\})$.
\end{proof}

An important property of $\K$-free algebras with $\L$-order is that they
satisfy exactly all inequalities which are satisfied by all algebras in
the variety $\K$. This pertains not only to inequalities which are satisfied
fully (i.e., to degree $1$) but to all inequalities which may be satisfied
to arbitrary degrees in $\L$. Indeed, as a consequence of the following lemma,
we obtain that $\val[\K]{t \ineq t'} = \val[\falg{\K}]{t \ineq t'}$
for all $t,t' \in T(X)$.

\begin{lemma}
  Let $\K$ be a variety. For any $t,t' \in T(X)$, we have
  \begin{align}
    \val[\K]{t \ineq t'} = (\fcon{\K})(t,t').
    \label{eqn:K_fcon}
  \end{align}
\end{lemma}
\begin{proof}
  Observe that the $\leq$-part of~\eqref{eqn:K_fcon} follows directly
  by the fact that $\falg{\K} \in \K$. In order to prove the $\geq$-part,
  take $\st \in \K$ and an $\st$-valuation $v\!: X \to M$. The valuation admits
  a homomorphic extension $\home{v}\!: \TX \to \st$ for which
  $\val{t \ineq t'} = \home{v}(t) \ineq^{\st} \home{v}(t') =
  \qc_{\home{v}}(t,t')$. Since 
  $\qc_{\home{v}} \in \fcset{\K}$, we get 
  $\fcon{\K} \subseteq \qc_{\home{v}}$ and thus
  $\val{t \ineq t'} \geq (\fcon{\K})(t,t')$,
  proving the $\geq$-part of~\eqref{eqn:K_fcon}.
\end{proof}

The following assertion shows that each variety of algebras with fuzzy
orders is an inequational class of
a theory over a denumerable set of variables. The assertion is a direct
analogy of the Bloom variety theorem~\cite{Bloom} of ordered algebras.

\begin{theorem}\label{thm:var_ineq}
  If $\K$ is a variety, then there is $\Sigma$ such that $\K = \Mod$.
\end{theorem}
\begin{proof}
  Let $X$ be a denumerable set of variables and put $\Sigma = \IneqX{X}$
  where $(\IneqX{X})(t \ineq t') = \val[\K]{t \ineq t'}$
  for all $t,t' \in T(X)$.
  By standard arguments, $\K \subseteq \Mod[\IneqX{X}] = \Mod$ and
  \begin{align*}
    \IneqX{X} = \IneqX[{\Mod[\IneqX{X}]}]{X} = \IneqX[\Mod]{X},
  \end{align*}
  i.e., $\val[\K]{t \ineq t'} = \val[\Mod]{t \ineq t'}$
  for all $t,t' \in T(X)$. Therefore, using~\eqref{eqn:K_fcon},
  we get $\fcon{\K} = \fcon{\Mod}$, i.e.,
  $\falg{\K}$ coincides with $\falg{\Mod}$.
  Hence, Lemma~\ref{le:V_falg} yields $\K = \Mod$.
\end{proof}

As a consequence of the previous observations, we get that varieties
are exactly inequational classes:

\begin{corollary}
  A class $\K$ of algebras with $\L$-orders is a variety if{}f
  $\K$ is an inequational class.
  \qed
\end{corollary}

In the rest of this paper, we comment of further closure properties
of inequational classes defined by theories using only particular degrees
in $\L$. The comments are inspired by an approach to algebras in
fuzzy setting proposed in~\cite{Haj:AnBvtfl} and which later appeared
in~\cite{CiHa:Tnbpfl}. The major difference between \cite{Haj:AnBvtfl}
and the Pavelka-style approaches (including the approach in the present paper)
is that \cite{Haj:AnBvtfl} allows to consider algebras with fuzzy equalities
using all BL-algebras taken as structures of truth degrees at the same time.
As a consequence, the approach uses equational theories in the usual sense
(ordinary sets of identities) and exploits the fact that varieties of
algebras with fuzzy equalities can be put in correspondence with varieties
of ordinary algebras by means of skeletons and homomorphic images.

Analogous characterization as in~\cite{Haj:AnBvtfl} can also be established
in our case with a fixed $\L$ as the structure of degrees. Indeed, for
each algebra $\st = \loalg$ with $\L$-order, let $\ske{\st}$ denote the
algebra with $\L$-order defined on $M$ with the same functions as in $\st$,
and with $\ineq^{\ske{\st}}$ being the identity $\L$-relation.
Note that each $\st$ is a homomorphic image of $\ske{\st}$ which
is called the \emph{skeleton} of $\st$. A theory $\Sigma$ is \emph{crisp}
if $\Sigma(t \ineq t') \in \{0,1\}$ for all $t,t' \in T(X)$.
We may now establish the following characterization.

\begin{theorem}\label{th:var_crisp}
  A class $\K$ of algebras with $\L$-orders
  is a variety closed under formations of sekeletons if{}f
  $\K = \Mod$ for crisp $\Sigma$.
\end{theorem}
\begin{proof}
  Obviously, if $\K = \Mod$ for crisp $\Sigma$, then $\K$
  is closed under formations of sekeletons. Conversely, let $\K$
  be a variety closed under skeletons.
  Observe that if $\val[\K]{t \ineq t'} < 1$,
  then there is $\st \in \K$ and $v\!: X \to M$ such that
  $\val{t \ineq t'} < 1$. Now, for the sekeleton $\ske{\st}$ of $\st$ and
  using the same valuation~$v$, we have $\val[\ske{\st},v]{t \ineq t'} = 0$.
  Therefore, $\val[\K]{t \ineq t'} = 0$ because $\ske{\st} \in \K$.
  As a consequence, $\IneqX{X}$ is crisp and we have $\K = \Mod$ for
  crisp $\Sigma = \IneqX{X}$, see Theorem~\ref{thm:var_ineq}.
\end{proof}

\begin{remark}
  Theorem~\ref{th:var_crisp} may further be generalized. For instance,
  for each $c < 1$, one may consider for $\st \in \K$ an algebra $\thr{c}$
  with the same universe and functions as $\st$ and with
  $\ineq^{\thr{c}}$ defined by
  \begin{align*}
    \e{a} \ineq^{\thr{c}} \e{b} &=
    \left\{
      \begin{array}{@{\,}l@{\quad}l@{}}
        1, &\text{if } \e{a} \ineq^{\st} \e{b} = 1, \\
        c, &\text{otherwise},
      \end{array}
    \right.
  \end{align*}
  for all $\e{a},\e{b} \in M$. Obviously, $\thr{c}$ is an algebra
  with $\L$-order. Using analogous arguments as in the proof of 
  Theorem~\ref{th:var_crisp}, we conclude that the following statements
  are equivalent:
  \begin{enumerate}
  \item[\itm{1}]
    $\K$ is a variety such that $\thr{c} \in \K$ whenever $\st \in \K$;
  \item[\itm{2}]
    $\K = \Mod$ where
    $\{\Sigma(t \approx t');\, t,t' \in T(X)\} \subseteq [0,c] \cup \{1\}$.
  \end{enumerate}
  Clearly, for $c = 0$ we obtain exactly the equivalence in
  Theorem~\ref{th:var_crisp}.
  Further investigation of closure properties of varieties of algebras
  with $\L$-order and the corresponding constraints on $\Sigma$ should
  prove interesting.
\end{remark}

\subsubsection*{Acknowledgment}
Supported by grant no. \verb|P202/14-11585S| of the Czech Science Foundation.


\footnotesize
\bibliographystyle{amsplain}
\bibliography{vtawfo}

\end{document}